\documentclass[10pt]{amsart}%

\usepackage{amsmath}
\usepackage{srcltx}
\usepackage{amsfonts}
\usepackage{amssymb}
\usepackage{psfrag}
\usepackage{verbatim}
\usepackage{graphicx}
\usepackage{color}
\providecommand{\U}[1]{\protect\rule{.1in}{.1in}}

\newtheorem{theorem}{Theorem}

\newtheorem{definition}[theorem]{Definition}

\newtheorem{lemma}[theorem]{Lemma}

\newtheorem{proposition}[theorem]{Proposition}
\newtheorem{remark}[theorem]{Remark}

\begin{document}
\title[Piecewise Continuous Chua System]{Existence of Closed Orbit in Piecewise Continuous Matsumoto-Chua System}
\author[J. Cassiano]{Jeferson Cassiano$^1$ }
\address{Centro de Matem\'atica Computa\c{c}\~ao e Cogni\c{c}\~ao. Universidade Federal do ABC, 09210-170. Santo Andr\'e. S.P.
Brazil}
\email{$^1$jeferson.cassiano@ufabc.edu.br}
\date{}
\dedicatory{}
\begin{abstract}
In this paper we study conditions for the existence of the close orbit in piecewise continuous Matsumoto-Chua System. Our interest are in a pseudo saddle-node bifurcation in the first return application. In this bifurcation the close orbit is releted to stable fixed point.  
\end{abstract}
\maketitle
\section{Introduction to continuous piecewise systems}
\label{section1}

In Control theory the nonsmooth motion phenomena have been described by equations defined for piecewise-continuous vector fields. The first results can be seen in \cite{Bar,Ut,Fil}.

A particular case of interest is the one modelled by sign functions which are mentioned in the literature as {\it relay systems} \cite{JT1}.

In this work we analyse the piecewise continuous Matsumoto-Chua System:

\begin{eqnarray}
\begin{array}
 [l]{l}
 \dot{x}_1=-\beta x_2\\
 \dot{x}_2=x_1-x_2+x_3\\
 \dot{x}_3=\alpha\left[x_2-x_3+sgn\left(x_3\right)\right]
 \end{array}
\label{modelo}
\end{eqnarray}

where $x_1,\ x_2$ and $x_3$ depend on $t$, $\alpha$ and $\beta$ are positive paramiters related to the physical model decribed later, and $sgn\left(.\right):\mathbb{R}^*\rightarrow\left\{-1,1\right\},x\mapsto\frac{\left|x\right|}{x}$ is the signal function.

In the literature, one can find many papers about oscillator dynamics, where discontinuous vector fields appear in a natural way \cite{Bernardo}. In this context, both analytical and numerical approaches are considered in the investigation of closed orbits \cite{ChinPhyLett}, chaotic behavior \cite{BifChaos} and sliding bifurcation \cite{Tex}.


Our goal is to study system \eqref{modelo} and derive conditions for some oscillatory behavior. After establishing \eqref{modelo} is obtained from a piecewise linear system, we study its singularities and tangencies observing some attracting regions of the phase space related to Filippov's procedure \cite{Fil}.

Our main theorem is stated as:
\begin{theorem}
At the parameter space $\left\{(\alpha,\beta):\alpha,\beta>0\right\}$, where
$\alpha=\frac{C_2}{C_1}$  and
$\beta=\frac{r^2C_2}{L}$, the system \eqref{modelo} admits a limit cycle if $\beta\left(1+20\alpha\right)<4\left[\alpha\left(1+\alpha\right)^3+\beta\left(\beta+2\alpha^2\right)\right]$ and:
\begin{itemize}
\item[(i)] $\beta<\frac{2\left(1+\alpha\right)^3}{9\left(1-2\alpha\right)}$, if $\alpha\in\left.\right]\frac{1}{8},\frac{1}{2}\left[\right.$; or 
\item[(ii)] $\alpha\geq\frac{1}{2}$.


\end{itemize}
\label{teocilinder}
\end{theorem}

This paper is organized in the following way: in section 2 we give some preliminaries, definitions and establish the notation used forward; in section 3 we describe the physical problem and perform a initial qualitative analysis; in section 4 we prove Theorem \ref{teocilinder} and in section 5 we present some numerical simulations related to the discovered bifurcation.
\section{Preliminaries}

In this section we introduce some of the terminology, basic
concepts and results that will be used in the sequel.

Let $H_0=\{x\in\mathbb{R}^n\,;\,h(x)=0\}$ a
non-singular hypersurface at $0\in\mathbb{R}^n.$ For
$q=\pm1$ denote by
$H_q=\{x\in\mathbb{R}^n\,;\,q h(x)>0\}$ and
$\bar{H}_q=H_q\cup H_0.$ We observe that $H_0$
represents the common boundary separating $H_q,$
$q=\pm1.$

Denote by $\mathcal{X}^r$ the set of all $\mathcal{C}^r$ vector
fields in $\mathbb{R}^n$ endowed with the $\mathcal{C}^r$ topology
with $r$ large enough for our purposes. Also denote by
$\mathcal{G}^r$ the set of vector fields $\textbf{f}$ on $\mathbb{R}^n$
given by:
\begin{equation}
\textbf{x}\in H_q\Rightarrow\textbf{f}(\textbf{x})=\textbf{f}_q(\textbf{x})
\end{equation}
where $\textbf{f}_q\in\mathcal{X}^r$ for $q=\pm1$ and
on $H_0$ the solution curves of $\textbf{f}$ obey Filippov's rules. We
refer the reader to \cite{Fil} for the mathematical
justification.

The vector field $\textbf{f}$ defined in the previous way is called {\it
discontinuous vector field} and the manifold $H_0$ is the {\it
discontinuity set} of $\textbf{f}$

We observe that $\textbf{f}\in\mathcal{G}^r$ is defined in $\mathbb{R}^n$
and $\textbf{f}_q$ is defined in $H_q.$ However
$\textbf{f}_q$ can be smoothly extended to a whole closed-half
space $\bar{H}_q$ for $q=\pm1.$ With this in
mind we will define the singularities of $\textbf{f}_q$ in $H_0$
(simple singularities) and the orbits of $\textbf{f}$ through $H_0$ (simple
orbits). Naturally, we denote by $\bar{\textbf{f}}_q$ the
extension of $\textbf{f}_q$ to $\bar{H}_q$. For simplicity,
we also denote $H_{\pm1}$ by $H_\pm$ and $\textbf{f}_{\pm1}$ by $\textbf{f}_\pm$.

\begin{definition}
We say that $\gamma:I\subset\mathbb{R}\rightarrow\mathbb{R}^n$, $t\mapsto\gamma(t)$, with $I$ an open
interval, is a simple solution of $\textbf{f}$ if:
\begin{itemize}
\item[(a)] $\gamma\in\mathcal{C}^0\left(I\right)$ and piecewise $\mathcal{C}^1\left(I\right)$;

\item[(b)] $\forall\ t\in I:\ \gamma(t)\in H_q,$
$\dot{\gamma}(t)=\left(\textbf{f}_q\circ\gamma\right)(t)$;

\item[(c)] $\gamma(I)\cap H_0$ is a discrete subset.
\end{itemize}
\end{definition}

Note that if $\gamma$ is a simple solution, then for all open
subinterval $\tilde{I}\subset I$ we have $\gamma(\tilde{I})\cap
H_0^c\neq\emptyset$ where $H_0^c$ is the complementary of $H_0$ in
$\mathbb{R}^n.$ So, simple solution or intersects $H_0$ at isolated
points or not intersects $H_0$ at all ($\mathcal{C}^r-$solution).

If $\textbf{f}_q\in\mathcal{X}^r$ and
$h\in\mathcal{C}^\infty$ we can define the function
$\mathcal{L}_{\textbf{f}_q}h:\mathbb{R}^n\rightarrow\mathbb{R}$ given by
$\mathcal{L}_{\textbf{f}_q}h(\textbf{x})=\left.\frac{d\ }{d\lambda}\right|_{\lambda=0}h\left(\textbf{x}+\lambda\textbf{f}_q\left(\textbf{x}\right)\right)$. Moreover, we can define
$\mathcal{L}^{k+1}_{\textbf{f}_q}h:\mathbb{R}^n\rightarrow\mathbb{R}$ by
$\mathcal{L}^{k+1}_{\textbf{f}_q}h=\mathcal{L}_{\textbf{f}_q}\left(\mathcal{L}^{k}_{\textbf{f}_q}h\right)$
for $k\geq 1$. The functions $\mathcal{L}^{k}_{\textbf{f}_q}h$ measure the contact between the
vector field $\textbf{f}_q$ and the manifold $H_0.$

\begin{definition}
We say that $\textbf{p}\in\mathbb{R}^n$ is {\it regular} if one of the
following conditions occurs:
\begin{itemize}
\item[(a)] $\textbf{p}\in H_0$ and $\mathcal{L}_{\bar{\textbf{f}}_q}h(\textbf{p})\mathcal{L}_{\bar{\textbf{f}}_{-q}}h(\textbf{p})\neq 0$;

\item[(b)] $\textbf{p}\in H_q$ and $\mathcal{L}_{\bar{\textbf{f}}_q}h(\textbf{p})\neq0.$
\end{itemize}
Otherwise, $\textbf{p}$ is called {\it singular} point.
\end{definition}

If all orbit points are regular we say this orbit is a {\it
regular} one. The orbits which contain singular points may be not well define (see
\cite{Fil}).

The points on the discontinuity set $H_0$ are distinguished as
follows:

\begin{definition}
Given $\textbf{p}\in H_0$ we say that
\begin{itemize}
\item[(a)] $\textbf{p}$ is in the {\it sewing region} if
$\mathcal{L}_{\bar{\textbf{f}}_q}h(\textbf{p})\mathcal{L}_{\bar{\textbf{f}}_{-q}}h(\textbf{p})>0.$ In this case we say that $\textbf{p}\in W$;

\item[(b)] $\textbf{p}$ is in the {\it escaping region} if
$q\mathcal{L}_{\bar{\textbf{f}}_q}h(\textbf{p})>0.$ In this case we
say that $\textbf{p}\in E$;

\item[(c)] $\textbf{p}$ is in the {\it sliding region} if
$q\mathcal{L}_{\bar{\textbf{f}}_q}h(\textbf{p})<0.$ In this case we
say that $\textbf{p}\in S$.
\end{itemize}
\end{definition}

The three regions: sewing, escaping and sliding are
open and disjoints ones and their frontiers are composed of
$\textbf{f}_q$ singular points. Next, we will
characterize some different types of $\textbf{f}_q$ singular points.
For a complete classification of discontinuous dynamical systems singular
points see \cite{JT}.

\begin{definition}
We say that a singular point $\textbf{p}$ of $\textbf{f}_q$ at $H_0$ is of:
\begin{itemize}
\item[(a)] fold type if $\mathcal{L}^2_{\bar{\textbf{f}}_q}h(\textbf{p})\neq 0$;

\item[(b)] cusp type if the set $\{\textbf{grad}\ h(\textbf{p}),\,\textbf{grad}\ 
\mathcal{L}_{\bar{\textbf{f}}_q}h(\textbf{p}),\,\textbf{grad}\ \mathcal{L}^2_{\bar{\textbf{f}}_q}h(\textbf{p})\}$ are linearly independent, $\mathcal{L}^2_{\bar{\textbf{f}}_q}h(\textbf{p})=0$ and $\mathcal{L}^3_{\bar{\textbf{f}}_q}h(\textbf{p})\neq 0$.
\end{itemize}
\end{definition}

The cusp points are isolated and they are found in the fold points curve extremes.

\subsection{Escaping Dynamics and Simple Singularities of $\textbf{f}$}.

In this subsection we will define, following Filippov's
convention, a new vector field in the escaping region, which 
is the subset of $H_0$ where
both vector fields $\textbf{f}_q$ and $q=\pm1$ point
toward $H_0$. We call it as {\it escaping
vector field} and will be denoted by $\textbf{f}_S$.

Consider $\textbf{f}\in\mathcal{G}^r$ and $\textbf{p}\in E$. A cone of the vectors \textbf{u} and \textbf{v} is defined by $\left(\textbf{u},\textbf{v}\right)=\left\{\textbf{w}:\textbf{w}=\mu\textbf{u}+\left(1-\mu\right)\textbf{v}, \mu\in\left(0,1\right)\right\}$. The escaping vector
field associated to $\textbf{f}$ is the smooth vector field $\textbf{f}_E$ tangent
to $H_0$ and defined at $\textbf{p}$ by $\textbf{f}_E(\textbf{p})$ such that $\left\{-\textbf{f}_E(\textbf{p})\right\}=T_{\textbf{p}}\left(H_0\right)\cap \left(-\bar{\textbf{f}}_{q}\left(\textbf{p}\right),-\bar{\textbf{f}}_{-q}\left(\textbf{p}\right)\right)$. Note that this definition is the dual definition of sliding vector field (see Figure
\ref{sliding}). So, $\textbf{f}_E$ is defined
in the open set $E$ with boundary $\partial E$.

\begin{remark}
If $\bar{\textbf{f}}_q(\textbf{p})$ and $\bar{\textbf{f}}_{-q}(\textbf{p})$ with $\textbf{p}\in E$
are linearly dependent, then $\textbf{f}_E(\textbf{p})=\textbf{0}.$ In this case we say that
$\textbf{p}$ is a pseudo-singularity of $\textbf{f}$.
\end{remark}

\begin{figure}[h!]
\centering
\includegraphics[scale=0.35]{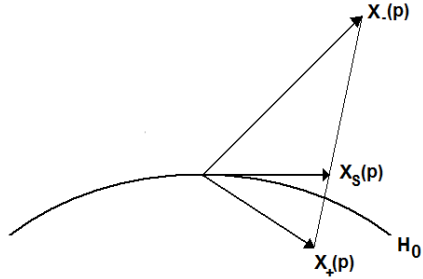}
\caption{Sliding vector field}\label{sliding}
\end{figure}

\begin{remark}
If $\textbf{p}\in S$ for the vector field $\textbf{f}$ then $\textbf{p}\in
E$ for the vector field $-\textbf{f}$. So, we can define the {\it
sliding vector field} on $S$ associated to $\textbf{f}$ by
$\textbf{f}_S=-\left(-\textbf{f}_E\right)$ and claim that $\textbf{f}_S$ has similar properties
than $\textbf{f}_S$.
\end{remark}

The next definition summarize the simple singularities of $\textbf{f}$ at
$H_0$.

\begin{definition}
An element $\textbf{p}\in H_0$ is classified
in the following way:
\begin{itemize}
\item[(a)] $\textbf{p}\in Int\left(E\right)$ {\rm(}resp.
$Int\left(S\right)${\rm)} and it is a critical point of $\textbf{f}_E$
{\rm(}resp. $\textbf{f}_S${\rm)}. In this case we say that $\textbf{p}$ is a
pseudo-singularity of $\textbf{f}.$

\item[(b)] $\textbf{p}$ is a tangency point between
$\partial S$ and $\textbf{f}_S$ or between
$\partial E$ and $\textbf{f}_E.$

\item[(c)] $\textbf{p}$ is a corner of $\partial S$ or
$\partial E.$
\end{itemize}

If $\textbf{p}$ is not a singularity of $\textbf{f}$ we refer to it as a {\it
regular point} of $\textbf{f}.$
\end{definition}

The next Lemma is a known result and its prove can be found in
\cite{Tex}.

\begin{lemma}
\begin{itemize}
\item[(a)] The escaping vector field $\textbf{f}_E$ is of class
$\mathcal{C}^r$ and it can be smoothly extended beyond the
boundary of $E$.

\item[(b)] If a point $\textbf{p}$ in $\partial E$ is a
fold point {\rm(}resp. cusp point{\rm)} of $\textbf{f}_{q}$ and an
$H_0$-regular point of $\textbf{f}_{-q}$ then $\textbf{f}_E$ is transverse
to $\partial E$ at $\textbf{p}$ (resp. $\textbf{f}_E$ has a
quadratic contact with $\partial E$ at \textbf{p}).
\end{itemize}\label{slidinglemma}
\end{lemma}

In this problem it is used a regulazation as it is defined above.
\begin{definition}
A $\mathcal{C}^{\infty}$ function $\varphi:\mathbb{R}\rightarrow\mathbb{R}$ is a transition function if $\forall x\in\mathbb{R}:\ qx\geq1\Rightarrow\varphi\left(x\right)=q,\ q=\pm1$; and $\forall x\in\left.\right]-1,1\left[\right.\Rightarrow\frac{d\varphi}{dx}\left(x\right)>0$. A $\varphi-$regularization is the vector field
\begin{equation}
\textbf{f}_{\epsilon}:M\rightarrow T\left(M\right),\textbf{x}\mapsto\displaystyle\sum_{q\in\left\{-1,1\right\}}\left[\frac{1+q\cdot\left(\varphi_{\epsilon}\circ h\right)\left(\textbf{x}\right)}{2}\right]\textbf{f}_q\left(\textbf{x}\right)
\end{equation}
being $\varphi_{\epsilon}\left(x\right)=\varphi\left(\frac{x}{\epsilon}\right),\ \epsilon>0$ and $h:M\rightarrow\mathbb{R}$ a regular function.
\end{definition}
The other definition is a so called germ.
\begin{definition}
Let $f:U\rightarrow M$ and $\tilde{f}:V\rightarrow M$ two functions ($U$ and $V$ are open sets) and let $x\in U\cap V$. So if $\exists U_x$ open set with $x\in U_x$ such that $f|_{U_x}=\tilde{f}|_{U_x}\Rightarrow f\sim\tilde{f}$, i.e., they are germ-equivalent in $x$. These equivalent classes are called \textbf{germs of functions in} $x$.
\end{definition}
So, in this case, $\textbf{f}\sim\textbf{f}_{\epsilon}\ \forall\textbf{x}\in\mathbb{R}^2\times\left(\left(-\infty,-\epsilon\right)\cup\left(\epsilon,+\infty\right)\right)$. 

\section{The system and its singularities}
\label{section2}

The Matsumoto-Chua's circuit has a big historical importance due, on the one hand, to its simplicity of a physical assembled as one can see in the figure \ref{fig:sistema} (it is a third order circuit) and, on the other hand, the phenomena's wealth in this dynamic. So, it has been object of intensive study. 

\begin{figure}[h]
\centering
\includegraphics[scale=0.35]{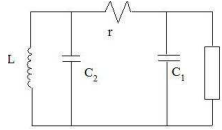}
\caption{The Matsumoto-Chua circuit}
\label{fig:sistema}
\end{figure}

It is presented in (\ref{dimensional}) the mathematical model of the Matsumoto-Chua's circuit. The circuit is composed by a black-box type device connected to a capacitor parallelly. Another parallel connection occurs between other capacitor and a inducer. These two meshes are connected in series with a resistor between them in one of the terminals and connect to the others, closing, so, the circuit. In literature is common refer to the black-box device as a nonlinear resistor. 

Let, so, the following system of ordinary differentials equations  
\begin{eqnarray}
 \begin{array}
 [l]{l}
 L\frac{di_L}{dt_d}=-v_{C_2}\\
 C_2\frac{dv_{C_2}}{dt_d}=\frac{v_{C_1}-v_{C_2}}{r}+i_L\\
 C_1\frac{dv_{C_1}}{dt_d}=\frac{v_{C_2}-v_{C_1}}{r}-i
 \end{array}
 \label{dimensional}
\end{eqnarray}
being $r$ electrical resistance, $t_d$ the dimensional time, $v_{C_1}$ and $v_{C_2}$ the differences in potential between the terminals of the capacitors of capacitance $C_1$ and $C_2$, $i_L$ the current through the inductor of inductance $L$ and $i$ the current through the black-box device.

Let $I$ a value of reference for a current. So, it define the variables
\begin{equation}
t=\frac{t_d}{rC_2},\ \ x_1=\frac{i_L}{I},\ \ x_2=\frac{v_{C_2}}{rI},\ \ x_3=\frac{v_{C_1}}{rI}
\end{equation}
and the parameters 
\begin{equation}
\alpha=\frac{C_2}{C_1},\ \ \beta=\frac{r^2C_2}{L}
\end{equation}
and $\phi=\frac{i}{I}$. Note that $\frac{1}{\alpha}$ and $\frac{1}{\beta}$ are capacitive ($C_1$) and inductive adimensional times constants, being $\tau_{C_i}=rC_i,\ i=1,2$ and $\tau_L=\frac{L}{r}$ the dimensional times constants. So, the equations system that describes a Matsumoto-Chua's circuit type is
\begin{eqnarray}
 \begin{array}
 [l]{l}
 \frac{dx_1}{dt}=-\beta x_2\\
 \frac{dx_2}{dt}=x_1-x_2+x_3\\
 \frac{dx_3}{dt}=\alpha\left[x_2-x_3-\phi\left(x_3\right)\right]
 \end{array}
 \label{adimensional}
\end{eqnarray}

In this paper it will be considered that 
\begin{equation}
i\left(v\right)=-I sgn\left(v\right)
\end{equation}
being $v$ difference in potential between the terminals from the black-box device and
\begin{equation}
sgn:\mathbb{R}^*\rightarrow\left\{-1,1\right\},x\mapsto\frac{\left|x\right|}{x}
\label{sinal}
\end{equation}
is called signal function. Note that the function (\ref{sinal}) is not defined in $0$ (it is here, making a comment that in some literatures continuously extend the function sign in $0$ to the right or to the left, but this extension is not necessary in this work). 

In this form the classical model can be seen as a first order approximation to a transition function in the regularization of the piecewise continuous vector fields in the Singular Pertubation Geometric Theory. The main question this theory is the approximation of the singular orbit by regular orbits in the Hausdorff distance. 

So 
\begin{equation}
\phi\left(x\right)=-sgn\left(x\right)
\label{caso}
\end{equation}
so piecewise $C^{\infty}$ (the semi-straight positive and negative). 

So, the Anosov's normal form of the vector field is
\begin{equation}
\textbf{f}:H_+\cup H_-\rightarrow T\left(H_+\cup H_-\right), \textbf{x}\mapsto\textbf{T}\left(\textbf{x}\right)+\alpha\textbf{b}sgn\left(\textbf{b}\cdot\textbf{x}\right)
\label{reduced}
\end{equation}
being $h=\pi_3$ being $\pi_3:\mathbb{R}^3\rightarrow\mathbb{R},\ \textbf{x}\mapsto\textbf{b}\cdot\textbf{x}$ projection and $\textbf{T}:H\rightarrow TH$ is a linear transformation, such that 
\begin{eqnarray}
\left[\textbf{T}\right]=\left[
\begin{array}
[c]{c c c}
0&-\beta &0\\
1&-1&1\\
0&\alpha &-\alpha
\end{array}
\right]\ \text{and}\ \left[\textbf{b}\right]=\left[
\begin{array}
[c]{c}
0\\
0\\
1
\end{array}
\right]\nonumber
\end{eqnarray}
in a canonical basis.
\begin{proposition}
The vector field (\ref{reduced}) is invariant by involution $R:H_q\rightarrow H_q,q=\pm1,\ \textbf{x}\mapsto -\textbf{x}$.
\end{proposition}
\begin{proof}
$R\textbf{f}\left(\textbf{x}\right)=-\left[\textbf{T}\left(\textbf{x}\right)+\alpha\textbf{b}sgn\left(\textbf{b}\cdot\textbf{x}\right)\right]=-\textbf{T}\left(\textbf{x}\right)-\alpha\textbf{b}sgn\left(\textbf{b}\cdot\textbf{x}\right)=$ $\textbf{T}\left(-\textbf{x}\right)+\alpha\textbf{b}\left[-sgn\left(\textbf{b}\cdot\textbf{x}\right)\right]=\textbf{T}\left(-\textbf{x}\right)+\alpha\textbf{b}sgn\left(-\textbf{b}\cdot\textbf{x}\right)=\textbf{T}\left(-\textbf{x}\right)+\alpha\textbf{b}sgn\left(\textbf{b}\cdot\left(-\textbf{x}\right)\right)=\textbf{f}\left(R\textbf{x}\right)$.
\end{proof}
This result show that this dynamical system is symmetric with relation to origin $\textbf{0}$. A other result is about stability of the singularities in $H_q$.

For this particular system we have the discontinuity region $H_0=\left\{\textbf{x}\in\mathbb{R}^3;\,x_3=0\right\}$, the connected components $H_q=\left\{\textbf{x}\in\mathbb{R}^3;\,qx_3>0\right\}$ and $\textbf{f}_q:H_q\rightarrow T\left(H_q\right),\ \textbf{x}\mapsto\textbf{T}\left(\textbf{x}\right)+q\alpha\textbf{b}$, $q=\pm 1$.

The next results describe the discontinuity region with respect to the sewing and escape regions and its singularities.

\begin{lemma} For system \eqref{reduced} $H_0$ is divided into the following way:

\begin{itemize}
\item[(a)] the sewing region given by $W=\left\{\textbf{x}\in
H_0;\,\left|x_2\right|>1\right\}$;

\item[(b)] the escaping region given by $E=\left\{\textbf{x}\in
H_0;\,\left|x_2\right|<1\right\}$;

\end{itemize}
Moreover, as $\mathcal{L}_{\bar{\textbf{f}}_-}h<\mathcal{L}_{\bar{\textbf{f}}_+}h\ \ \forall\alpha>0$, there is no sliding region.
\end{lemma}

\begin{proof}
The system \eqref{reduced} satisfies, for $q=\pm1,$ $\mathcal{L}_{\bar{\textbf{f}}_q}h\left(\textbf{x}\right)=\alpha\left(x_2+q\right)$ being $\bar{\textbf{f}}_q$ extension of $\textbf{f}_q$ to $H_0$. We have that $\mathcal{L}_{\bar{\textbf{f}}_+}h\left(\textbf{x}\right)\mathcal{L}_{\bar{\textbf{f}}_-}h\left(\textbf{x}\right)=\alpha^2\left(x^2_2-1\right)$. As the sewing region is given by $\mathcal{L}_{\bar{\textbf{f}}_+}h\left(\textbf{x}\right)\mathcal{L}_{\bar{\textbf{f}}_-}h\left(\textbf{x}\right)>0$ and the escaping region is given by $q\mathcal{L}_{\bar{\textbf{f}}_q}h\left(\textbf{x}\right)=\alpha\left(qx_2+1\right)>0$, we obtain statements (a) and (b). Now it is easy to see that $\mathcal{L}_{\bar{\textbf{f}}_-}h\left(\textbf{x}\right)<\mathcal{L}_{\bar{\textbf{f}}_+}h\left(\textbf{x}\right)$. So there is no sliding region.
\end{proof}
Of the definition of the escaping region $E$, we note that $\forall\textbf{x}\in H,\ \forall t>0\Rightarrow\varphi_t\left(\textbf{x}\right)\notin E$. So, every minimal sets are disjoint to $E$.
\begin{proposition} The typical singularities in $H_0$ lie in the sliding boundary $\partial S=\left\{\textbf{x}\in H_0:\left|x_2\right|=1\right\}=\mathcal{L}_{\bar{\textbf{f}}_q}h^{-1}\left(\textbf{0}\right).$ Moreover we have $\forall\alpha>0$:

\renewcommand{\labelitemi}{} 

   \begin{itemize}
   \item if $qx_1>-1$ then $\textbf{f}_q$ has a hiperbolic fold singularity at $x_2=-q$;
   \item if $x_1=-q$ then $\textbf{f}_q$ has a hybrid cusp singularity at $\textbf{x}^c_q=-q(1,1,0)$;
	 \item if $qx_1<-1$ then $\textbf{f}_q$ has a elliptic fold singularity at $x_2=-q$.
   \end{itemize}

\label{typicalsing}
\end{proposition}

\begin{proof}
Let $\textbf{x}\in H_0.$ As $\mathcal{L}_{\bar{\textbf{f}}_q}h\left(\textbf{x}\right)=\alpha\left(x_2+q\right)$ for $q=\pm1$ are given by the points of $H_0$ satisfying $\mathcal{L}_{\bar{\textbf{f}}_q}h\left(\textbf{x}\right)=0$ we get $S_q=\left\{\textbf{x}\in H_0:x_2=-q\right\}$ as stated at the proposition. Now, if $\textbf{x}\in S_q$ then  $\mathcal{L}^2_{\bar{\textbf{f}}_q}h\left(\textbf{x}\right)=\alpha\left(x_1+q\right)$ and  
$\mathcal{L}^3_{\bar{\textbf{f}}_q}h\left(\textbf{x}^c_q\right)=q\alpha\beta\neq0$ $\forall\alpha,\beta.$ In 
order to $x$ be a fold singularity we must have  $\mathcal{L}^2_{\bar{\textbf{f}}_q}h\left(\textbf{x}\right)\neq0.$ This implies in $x_1\neq-q$. For a $2E-$fold it has that $q\left.\mathcal{L}^2_{\bar{\textbf{f}}_{q}}h\right|_{\partial S}=\alpha\left(qx_1+1\right)<0\Rightarrow qx_1<-1$ and for a $2H-$fold it has that $q\left.\mathcal{L}^2_{\bar{\textbf{f}}_{q}}h\right|_{\partial S}=\alpha\left(qx_1+1\right)>0\Rightarrow qx_1>-1$.

In the same way $\textbf{x}$ will be a cusp singularity if $x_1=-q$, $\mathcal{L}^3_{\bar{\textbf{f}}_q}h\left(\textbf{x}^c_q\right)\neq0$ and $\left\{\textbf{grad}\left(h\right),\textbf{grad}\left(\mathcal{L}_{\bar{\textbf{f}}_{q}}h\right),\textbf{grad}\left(\mathcal{L}^2_{\bar{\textbf{f}}_{q}}h\right)\right\}$ must be linearity independent. In the canonical basis this three vectors are presented by $\left[0\ 0\ 1\right]^t$, $\alpha\left[0\ 1\ -1\right]^t$ and $-\alpha\left[1\ -\left(1+\alpha\right)\ 1+\alpha\right]^t$ in $\textbf{x}^c_q$. It is easy to see that the vectors are independent, for exemple, note that the determinant of the associated matrix is $-\alpha^2$.

To show that the cusp is hybrid just showing that for $q\mathcal{L}^2_{\bar{\textbf{f}}_{-q}}h\left(\textbf{x}^c_q\right)=-2\alpha<0$ 
\end{proof}

Using Filippov's convention, we are able to find the escaping vector field defined in $E\subset H_0$. This two dimensional vector field is stated in the next proposition.

\begin{proposition} The escaping vector field $\textbf{f}^-:E\rightarrow T\left(E\right)$ is given by:
\begin{equation}
\textbf{x}\mapsto-\beta x_2\frac{\partial\ \ }{\partial x_1}+\left(x_1-x_2\right)\frac{\partial\ \ }{\partial x_2}
\label{campo_H0}
\end{equation}\label{slidingVC}
\end{proposition}

\begin{proof}
The vector field $\textbf{f}^-$ is defined as $\textbf{f}^-\left(\textbf{x}\right)=\mu\bar{\textbf{f}}_{-q}+\left(1-\mu\right)\bar{\textbf{f}}_{q}$ where $\mu\in [0,1]$ is such that $\mathcal{L}_{\textbf{f}^-}h\left(\textbf{x}\right)=0$ for every $x\in S.$ So, we obtain $\mu=\dfrac{\mathcal{L}_{\bar{\textbf{f}}_{q}}h\left(\textbf{x}\right)}{\mathcal{L}_{\bar{\textbf{f}}_{q}-\bar{\textbf{f}}_{-q}}h\left(\textbf{x}\right)}.$ Now, for system \eqref{reduced}, we have $\mathcal{L}_{\bar{\textbf{f}}_{q}}h\left(\textbf{x}\right)=\alpha\left(x_2+q\right)$ and $\mathcal{L}_{\bar{\textbf{f}}_{q}-\bar{\textbf{f}}_{-q}}h\left(\textbf{x}\right)=2q\alpha$. So $\mu=\dfrac{1+qx_2}{2}$. Substituting this expression into $\textbf{f}^-\left(\textbf{x}\right)=\mu\bar{\textbf{f}}_{-q}+\left(1-\mu\right)\bar{\textbf{f}}_{q}$, we obtain the result.
\end{proof}

The next proposition describes the singularities of the vector fields $\textbf{f}_q$ and $\textbf{f}^-$.

\begin{proposition} The $H_q-$singularities of system \eqref{reduced} are real and they are $\textbf{x}^*_q=-q\alpha\textbf{T}^{-1}\left(\textbf{b}\right)$ such that $\left[\textbf{x}^*_q\right]=q\left[-1\ 0\ 1\right]^t$, $q=\pm 1$. This singulatities are assintotically stables $\forall\alpha,\beta>0$. 
\end{proposition} 

\begin{proof}

Note that $\left|\textbf{T}\right|=-\alpha\beta<0\ \forall\alpha,\beta>0\Rightarrow\exists\textbf{T}^{-1}$. From system \eqref{reduced}, we get $\textbf{f}_q\left(\textbf{x}^*_q\right)=\textbf{T}\left(\textbf{x}^*_q\right)+q\alpha\textbf{b}=\textbf{0}$, for $q=\pm1$. This singulatities are real because $qx^*_3=q^2=1>0$. 

It is easy to show that the linear vector field associated to \ref{reduced} is $\textbf{T}:T_{\textbf{x}^*_q}\left(H_q\right)\rightarrow T_{\textbf{x}^*_q}\left(H_q\right),\ \textbf{y}\mapsto\textbf{T}\left(\textbf{y}\right)$ and the semigroup induced by this vector field is $\left\{e^{t\textbf{T}},\ t\geq0\right\}$. 

Now, let the characteristic polynomial associated to operator $\textbf{T}$
\begin{equation}
p_{\textbf{T}}\left(\lambda\right)=-\left[\lambda^3+\left(1+\alpha\right)\lambda^2+\beta\lambda+\alpha\beta\right]
\label{pol_char}
\end{equation}
So it makes use of the Routh's stability criterion in $\textbf{T}$, where the coeficients are given by $1,\alpha+1,\frac{\beta}{\alpha+1},\alpha\beta>0$, therefore implying in the semigroup to be a contraction. 
\end{proof}

The next proposition show the singularity in the escaping vector field \eqref{campo_H0}.

\begin{proposition} 
The escaping vector field \eqref{campo_H0} admits a assintotically stable singularity in $\textbf{0}$.
\end{proposition}

\begin{proof}
$\textbf{f}^-:H_0\rightarrow H_0 $ is a linear operator. So, $\left|\textbf{f}^-\right|=\beta\Rightarrow\exists\left(\textbf{f}^-\right)^{-1}$. The eigenvalues of the $\textbf{f}^-$ have real part $Re\left(-\frac{1}{2}\pm\sqrt{\left(\frac{1}{2}\right)^2-\beta}\right)<0\ \ \forall\beta>0$ and the system is assintotically stable in $\textbf{0}$.
\end{proof}

The results above are more important to prove of the theorem.

\section{Proof of Theorem \ref{teocilinder}}
\label{section4}

In this section we proof Theorem \ref{teocilinder}. First, we consider a Poincar\'{e} section in $H_0$ and a first return map to this section. We will applied a uniform contraction in $\textbf{x}$. So, we make a regularization and a linear blowing up. For this auxiliar system, we prove the existence of a minimal set type periodic orbits with the Banach Fixed Point Theorem in first return map. So, there is a fixed point associated to the limit cycle in first return map in this system and the discontinuous system. 

Let the parameter space $\left\{\left(\alpha,\beta\right)\in\mathbb{R}^2:\alpha,\beta>0\right\}$. We are to study the bifurcations.

We know that the singularities in $H_q$ are assintotically stable $\forall\alpha,\beta>0$.   

\begin{proposition}
Let $\textbf{y}=\epsilon\textbf{x}$, such that $0<\epsilon<<1$. So, the vector field \ref{reduced} is
\begin{equation} 
\textbf{y}\mapsto\textbf{T}\left(\textbf{y}\right)+\epsilon\alpha\textbf{b}\ sgn\left(\textbf{b}\cdot\textbf{y}\right)
\label{normal}
\end{equation}
\label{propRelative}
\end{proposition}

\begin{proof}
We have that $\dot{\textbf{x}}=\textbf{T}\left(\textbf{x}\right)+\alpha\textbf{b}\ sgn\left(\textbf{b}\cdot\textbf{x}\right)$ and $\textbf{y}=\epsilon\textbf{x}$. So, $\dot{\textbf{y}}=\textbf{T}\left(\textbf{y}\right)+\epsilon\alpha\textbf{b}\ sgn\left(\frac{\textbf{b}\cdot\textbf{y}}{\epsilon}\right)$. Note that $sgn\left(\frac{x}{\epsilon}\right)=sgn\left(x\right)\ \forall\epsilon>0,\ \forall x\in\mathbb{R}$. It is proved.
\end{proof}
Now, we make a $\varphi-$regularization in \ref{normal}. The regularized vector field is
\begin{equation}
\textbf{f}_{\epsilon}:H\rightarrow T\left(H\right),\ \textbf{y}\mapsto\textbf{T}\left(\textbf{y}\right)+\epsilon\alpha\textbf{b}\varphi_{\epsilon}\left(\textbf{b}\cdot\textbf{y}\right)
\label{regular}
\end{equation}

Let the operator $\phi:H\rightarrow H$ such that
\begin{eqnarray}
\left[\phi\right]=\left[
\begin{array}
[c]{c c c}
1&0&0\\
0&1&0\\
0&0&\epsilon
\end{array}
\right]\nonumber
\end{eqnarray}
So, we have a next result with the linear blowing up.

\begin{proposition}
The equivalent slow vector field to \eqref{regular} is
\begin{equation}
\left(\textbf{f}_{\epsilon}\circ\phi\right)\left(\textbf{u}\right)=\left(\textbf{T}\circ\phi\right)\left(\textbf{u}\right)+\epsilon\alpha\textbf{b}\varphi\left(\textbf{b}\cdot\textbf{u}\right)
\label{slow}
\end{equation}
\end{proposition}

\begin{proof}
Let $\textbf{y}=\phi\left(\textbf{u}\right)$. So, $\Rightarrow\left(\textbf{f}_{\epsilon}\circ\phi\right)\left(\textbf{u}\right)=\left(\textbf{T}\circ\phi\right)\left(\textbf{u}\right)+\epsilon\alpha\textbf{b}\varphi_{\epsilon}\left(\textbf{b}\cdot\phi\left(\textbf{u}\right)\right)=\left(\textbf{T}\circ\phi\right)\left(\textbf{u}\right)+\epsilon\alpha\textbf{b}\varphi_{\epsilon}\left(\phi^*\left(\textbf{b}\right)\cdot\textbf{u}\right)$ being $\phi^*$ the adjoint operator of the $\phi$. Note that $\phi^*\left(\textbf{b}\right)=\epsilon\textbf{b}\Rightarrow\varphi_{\epsilon}\left(\phi^*\left(\textbf{b}\right)\cdot\textbf{u}\right)=\varphi\left(\frac{\phi^*\left(\textbf{b}\right)\cdot\textbf{u}}{\epsilon}\right)=\varphi\left(\frac{\epsilon\textbf{b}\cdot\textbf{u}}{\epsilon}\right)=\varphi\left(\textbf{b}\cdot\textbf{u}\right)$
\end{proof}

Making $t=\epsilon\tau$ it has the fast vector field of the next proposition
\begin{proposition}
The equivalent fast vector field to \eqref{regular} is 
\begin{equation}
\textbf{g}\left(\textbf{u}\right)=\epsilon\left(\phi^{-1}\circ\textbf{f}_{\epsilon}\circ\phi\right)\left(\textbf{u}\right)=\epsilon\left[\textbf{T}\left(\textbf{u}\right)+\alpha\textbf{b}\varphi\left(\textbf{b}\cdot\textbf{u}\right)\right]
\label{fast}
\end{equation}
\end{proposition}
\begin{proof}
$t=\epsilon\tau\Rightarrow\epsilon\frac{d\textbf{u}}{dt}=\frac{d\textbf{u}}{d\tau}$. So, 
\begin{equation}
\epsilon\left(\phi^{-1}\circ\textbf{f}_{\epsilon}\circ\phi\right)\left(\textbf{u}\right)=\epsilon\left[\left(\phi^{-1}\circ\textbf{T}\circ\phi\right)\left(\textbf{u}\right)+\epsilon\alpha\phi^{-1}\left(\textbf{b}\right)\varphi\left(\textbf{b}\cdot\textbf{u}\right)\right]\nonumber
\end{equation}
Note that $\phi^{-1}\circ\textbf{T}\circ\phi=\textbf{T}$ and $\phi^{-1}\left(\textbf{b}\right)=\frac{\textbf{b}}{\epsilon}$. 
\end{proof}

We have that
\begin{eqnarray}
\epsilon\left[\textbf{T}\right]_{\phi}=\epsilon\left[\phi^{-1}\circ\textbf{T}\circ\phi\right]=\left[
\begin{array}
[c]{c c c}
0&-\epsilon\beta&0\\
\epsilon&-\epsilon&\epsilon^2\\
0&\alpha&-\epsilon\alpha
\end{array}
\right]\nonumber
\end{eqnarray}
Now we are in condition to prove the main result.

\begin{proof}[Proof of Theorem \ref{teocilinder}]
First, let the real root of the (\ref{pol_char}). So, in next proposition there is a estimate to the real root.  
\begin{proposition}
$\exists\lambda^*\in\left.\right]-\left(1+\alpha\right),-\alpha\left[\right.:\ \lambda^*\in Spec\left(\textbf{T}\right)$.
\end{proposition}
\begin{proof}
$p_{\textbf{T}}\left(-\alpha\right)=-\alpha^2<0$ and $p_{\textbf{T}}\left(-\alpha-1\right)=\beta>0$. Polynomial is continous in $\mathbb{R}$, so, by intermediary value theorem, $\exists\lambda^*\in\left.\right]-\left(1+\alpha\right),-\alpha\left[\right.:\  p_{\textbf{T}}\left(\lambda^*\right)=0$.
\end{proof}
Note that if $Spec\left(\textbf{T}\right)=\left\{\lambda^*\right\}\subset\mathbb{R}:\lambda^*$ is a simple root of $p_{\textbf{T}}$  $T_{x^*_q}\left(H_q\right)=\ker\left(\textbf{T}-\lambda^*\textbf{I}\right)\oplus\mathbb{W}:\textbf{T}\left(\mathbb{W}\right)=\mathbb{W}$ irredutible and $\dim\left(\mathbb{W}\right)=2$, so $\exists\textbf{x}\in W_q=\left\{\textbf{x}\in H_0:qx_2<-1\right\},\exists t>0:\ \varphi_t\left(\textbf{x}\right)\in W_{-q}$ that it is essential to first return map. 
\begin{proposition}
$\beta\left(1+20\alpha\right)<4\left[\alpha\left(1+\alpha\right)^3+\beta\left(\beta+2\alpha^2\right)\right]\Rightarrow Spec\left(\textbf{T}\right)=\left\{\lambda^*\right\}$.
\label{raiz}
\end{proposition}
\begin{proof}
In fact, the real critical points of $p_{\textbf{T}}$, if they exist, are given by
\begin{equation}
\lambda^*_{1,2}=-\left(\frac{\alpha+1}{3}\right)\pm\sqrt{\left(\frac{\alpha+1}{3}\right)^2-\frac{\beta}{3}}
\label{critical}
\end{equation}
and $\beta\left(1+20\alpha\right)<4\left[\alpha\left(1+\alpha\right)^3+\beta\left(\beta+2\alpha^2\right)\right]\Leftrightarrow p_{\textbf{T}}\left(\lambda^*_1\right)p_{\textbf{T}}\left(\lambda^*_2\right)>0$. Note that if (\ref{critical}) is not real value, so $\lambda^*_2=\bar{\lambda}^*_1\Rightarrow p_{\textbf{T}}\left(\lambda^*_1\right)p_{\textbf{T}}\left(\lambda^*_2\right)>0$. We have that $\beta\left(1+20\alpha\right)<4\left[\alpha\left(1+\alpha\right)^3+\beta\left(\beta+2\alpha^2\right)\right]\Rightarrow\lambda^*_2=\bar{\lambda}^*_1$ or the local extreme values have a same signals and $p_{\textbf{T}}$ has a only real root. But, if $\lambda^*_{1,2}\notin\mathbb{R}\Rightarrow p_{\textbf{T}}$ has a only real root.
\end{proof}
In the case of the $Spec\left(\textbf{T}\right)=\left\{\lambda^*\right\}$, the irreducible polynomial is
\begin{equation}
p\left(\lambda\right)=\frac{p_{\textbf{T}}\left(\lambda\right)}{\lambda-\lambda^*}=-\left[\lambda^2+\left(1+\alpha+\lambda^*\right)\lambda+\left(1+\alpha+\lambda^*+\beta\right)\lambda^*\right]
\label{irredutible}
\end{equation} 
The condition above is only necessary to existence of first return map. Now, supose that we have a first return map $P:\Omega\subset H_0\rightarrow\Omega$. We need to proof that $P$ is a contraction. First, we have that the flow in $H_q$ is $\varphi^q_t:H_q\rightarrow H_q,\ \textbf{x}\mapsto\textbf{x}^*_q+e^{t\textbf{T}}\left(\textbf{x}-\textbf{x}^*_q\right)$. It is easy to proof that $\varphi^q_t$ is a contraction $\forall t>0$. Note that
\begin{equation}
\left\|\varphi^q_t\left(\textbf{x}\right)-\varphi^q_t\left(\textbf{y}\right)\right\|=\left\|e^{t\textbf{T}}\left(\textbf{x}-\textbf{y}\right)\right\|\ \forall\textbf{x},\textbf{y}\in H_q
\end{equation}
and $e^{t\textbf{T}}$ is a contration $\forall t>0$, as before proof.
Let the next function: $t_q:\Omega\rightarrow\mathbb{R}^*_+,\textbf{x}\mapsto\displaystyle\min_{t>0}\left\{t\in\mathbb{R}:\left(h\circ\varphi^q_t\right)\left(\textbf{x}\right)=0\right\}$ $=\displaystyle\min_{t>0}\left\{t\in\mathbb{R}:h\left[e^{\textbf{A}t}\left(\textbf{x}-q\textbf{A}^{-1}b\right)\right]+q=0\right\}$. Note that the existence of the above function is related with the map $P$ and the flow can be extended to $\bar{H}_q$ (see definition of simple solution).        
\begin{definition}
The map $\varphi_q:\Omega\rightarrow\Omega'\subset H_0,\ \textbf{x}\mapsto\varphi^q_{t_q\left(\textbf{x}\right)}\left(\textbf{x}\right)$ is called semi-Poincar\'{e} map \cite{Hirsch:2003}.
\end{definition}
As $t_q\left(\textbf{x}\right)>0\Rightarrow$ the semi-Poincar\'{e} map $\varphi_q$ is a contraction. Now, we can to define the First Return Map. 
\begin{definition}
The map $P:\Omega\rightarrow\Omega,\ \textbf{x}\mapsto\left(\varphi_{-q}\circ\varphi_{q}\right)\left(\textbf{x}\right)$ is called First Return Map. The map $\varphi_{-q}:\Omega'\rightarrow\Omega$ is analogue to $\varphi_{q}$.
\end{definition}
This conditions above are not enough to existence of $P$. For these it is necessary that $P\left(\Omega\right)\subset\Omega$. So, we find the enough contidions for existence of the stable closed orbit in fast system (\ref{fast}). This orbit imply in stable fixed point in first return map of the (\ref{fast}). So, there is a stable fixed point in first return map of the original system. 
\begin{theorem}
At the parameter space $\left\{(\alpha,\beta):\alpha,\beta>0\right\}$ the system \eqref{fast} admits a limit cycle if $\beta\left(1+20\alpha\right)<4\left[\alpha\left(1+\alpha\right)^3+\beta\left(\beta+2\alpha^2\right)\right]$ and:
\begin{itemize}
\item[(i)] $\beta<\frac{2\left(1+\alpha\right)^3}{9\left(1-2\alpha\right)}$, if $\alpha\in\left.\right]\frac{1}{8},\frac{1}{2}\left[\right.$; or 
\item[(ii)] $\alpha\geq\frac{1}{2}$.
\end{itemize}
\label{theoremRegular}
\end{theorem}
\begin{proof}
Let the fast vector field (\ref{fast}). $H=H^{\epsilon}_+\cup H^{\epsilon}_0\cup H^{\epsilon}_-$ being $H^{\epsilon}_q=\left\{\textbf{u}\in H:\ qx_3\geq1\right\},$ $q=\pm1$ and $H^{\epsilon}_0=\left\{\textbf{u}\in H:\ \left|x_3\right|<1\right\}$ being $x_3=\textbf{b}\cdot\textbf{u}$.

In $H^{\epsilon}_0$ the vector field (\ref{fast}) is $\textbf{g}\left(\textbf{u}\right)=\alpha y_2\textbf{b}+\textbf{o}\left(\epsilon\right)$. The flow is
\begin{equation} 
\varphi^{\epsilon}_{\tau}\left(\textbf{u}\right)=\textbf{u}+\alpha y_2\tau\textbf{b}+\textbf{o}_1\left(\epsilon\right)\ by\ y_2>>\epsilon
\label{flowH0}
\end{equation}
that is a perturbation of a translaction in direction of $\textbf{b}$. 

In $\bar{H}^{\epsilon}_q=H^{\epsilon}_q\cup H^{\epsilon}_0$ we have singularities in $\textbf{u}^*_q=\epsilon\phi^{-1}\left(\textbf{x}^*_q\right)\in \partial H^{\epsilon}_q=\left\{\textbf{u}\in H:\ x=q\right\}$ such that $\left[\textbf{u}^*_q\right]=q\left[-\epsilon\ 0\ 1\right]^t$. The nature of stability is the same of (\ref{reduced}) because $\epsilon\phi^{-1}$ is a diffeomorphism.

Note that $\mathcal{L}_{\textbf{g}\left(\textbf{u}\right)}h|_{\partial H^{\epsilon}_q}=\alpha y_2$. So, the orientation of the vector field $\textbf{g}$ is positive to $y_2>0$ and negative to $y_2<0$. Let $\partial H^{\epsilon}_0=\partial H^{\epsilon}_+\cup\partial H^{\epsilon}_-$. So, the orientation of the vector field $\textbf{g}$ in $\partial H^{\epsilon}_0$ is positive to $qy_2>0$ or negative to $qy_2<0$.

Let $B=\left\{\textbf{u}\in\partial H^{\epsilon}_0:\ y_2=0\right\}$ the geometric place of the tangency of the vector field $\textbf{g}$ in $\partial H^{\epsilon}_0$. $\mathcal{L}^2_{\textbf{g}\left(\textbf{u}\right)}h|_B=\epsilon\left(y_1+q\epsilon\right)$. We know that if $q\mathcal{L}^2_{\textbf{g}\left(\textbf{u}\right)}h|_B>0\Leftrightarrow qy_1>-\epsilon$ the contact is hiperbolic and $q\mathcal{L}^2_{\textbf{g}\left(\textbf{u}\right)}h|_B<0\Leftrightarrow qy_1<-\epsilon$ the contact is elliptic. Let $\partial^+H^{\epsilon}_q=\left\{\textbf{u}\in\partial H^{\epsilon}_q:\ qy_2>0\right\}$ and $\partial^-H^{\epsilon}_q=\left\{\textbf{u}\in\partial H^{\epsilon}_q:\ qy_2<0\right\}$. It is possible to define the injective map $\varphi^q:D\subset\partial^+H^{\epsilon}_q\rightarrow\partial^-H^{\epsilon}_q,\ \textbf{u}\mapsto\varphi^{\epsilon}_{\tau}\left(\textbf{u}\right)$ for some $\tau>0$ such that $x_3=q$. Note that $p_{\epsilon\textbf{T}}\left(\lambda\right)=\epsilon^3p_{\textbf{T}}\left(\frac{\lambda}{\epsilon}\right)$ and $Spec\left(\epsilon\textbf{T}\right)=\left\{\epsilon\lambda^*\right\}$ if there is the conditions of \ref{raiz}.

Let $\lambda_{1,2}\in\mathbb{C}:\ p\left(\lambda_{1,2}\right)=0$ ($p$ is irredutible). We know, of \cite{Arnold}, that
\begin{itemize}
\item[(i)] $D\subsetneq\partial^+H^{\epsilon}_q$ if $\frac{Re\left(\lambda_{1,2}\right)}{\lambda^*}>1$ or;
\item[(ii)] $\varphi^q:\partial^+H^{\epsilon}_q\rightarrow\partial^-H^{\epsilon}_q$ is a bijection if $\frac{Re\left(\lambda_{1,2}\right)}{\lambda^*}=1$ or
\item[(iii)] $\varphi^q\left(\partial^+H^{\epsilon}_q\right)\subsetneq\partial^-H^{\epsilon}_q$ if $0<\frac{Re\left(\lambda_{1,2}\right)}{\lambda^*}<1$
\end{itemize}
Note that the case (iii) is necessary by first return map in fast system. So $\frac{Re\left(\lambda_{1,2}\right)}{\lambda^*}=\frac{1}{2}\left[\frac{1+\alpha}{\left(-\lambda^*\right)}-1\right]<\frac{1}{2}\left(\frac{1+\alpha}{\alpha}-1\right)=\frac{1}{2\alpha}$. If $\frac{1}{2\alpha}\leq1\Leftrightarrow\alpha\geq\frac{1}{2}\Rightarrow$ (iii). For $\alpha<\frac{1}{2}\Rightarrow\frac{1}{2}\left[\frac{1+\alpha}{\left(-\lambda^*\right)}-1\right]<1\Rightarrow\lambda^*<-\left(\frac{1+\alpha}{3}\right)$. Note that $\beta\left(\lambda\right)=-\lambda^2\left(\frac{1+\alpha+\lambda}{\alpha+\lambda}\right)$ is strictly increasing in $\left.\right]-\left(1+\alpha\right),-\alpha\left[\right.$. So $\beta\left(\lambda^*\right)<\beta\left(-\left(\frac{1+\alpha}{3}\right)\right)=\frac{2}{9}\frac{\left(1+\alpha\right)^3}{\left(1-2\alpha\right)}$. The curve of bifurcation $\beta=\frac{2}{9}\frac{\left(1+\alpha\right)^3}{\left(1-2\alpha\right)}$ has sense to triple root of $p_{\textbf{T}}$. This root is $\lambda^*=-\frac{3}{8}$ and for this one we have that $\alpha=\frac{1}{8}$ and $\beta=\frac{27}{64}$. So, for $\alpha<\frac{1}{8}$ there are only real roots or $\frac{Re\left(\lambda_{1,2}\right)}{\lambda^*}>1$. We have that 
\begin{equation}
\left\|\varphi^{\epsilon}_{\tau}\left(\textbf{u}\right)-\varphi^{\epsilon}_{\tau}\left(\textbf{v}\right)\right\|=\left\|e^{\epsilon\tau\textbf{T}}\left(\textbf{u}-\textbf{v}\right)\right\|=\left\|e^{t\textbf{T}}\left(\textbf{u}-\textbf{v}\right)\right\|\ \forall\textbf{u},\textbf{v}\in H^{\epsilon}_q
\end{equation}
So, $\varphi^{\epsilon}_{\tau}$ is a contration and $\varphi^{q}$ is a contration. 
It is possible to define a map with (\ref{flowH0}). Let $\Delta\tau=\frac{2q}{\alpha y_2}+o\left(\epsilon\right)$. So 
\begin{equation}
\varphi^0_q:\partial^-H^{\epsilon}_{-q}\rightarrow\partial^+H^{\epsilon}_{q},\ \textbf{u}\mapsto\textbf{u}+2q\textbf{b}+\textbf{o}_2\left(\epsilon\right)
\end{equation}
Now, we can to define a first return map $P_{\epsilon}:\Omega_{\epsilon}\rightarrow\Omega_{\epsilon},\ \textbf{u}\mapsto\left(\varphi^0_{-q}\circ\varphi^{-q}\circ\varphi^0_q\circ\varphi^q\right)\left(\textbf{u}\right)$ being $\Omega_{\epsilon}=\left\{\textbf{u}\in\partial^+H^{\epsilon}_q:\ \left|y_2\right|\geq\epsilon_0>>\epsilon\right\}$. $P_{\epsilon}$ is a contraction and $\mathbb{R}^2$ is complete, so, by Banach Fixed Point, $\exists !\textbf{u}^*\in\Omega_{\epsilon}:\ P_{\epsilon}\left(\textbf{u}^*\right)=\textbf{u}^*$ stable, that is, $\displaystyle\lim_{n\rightarrow+\infty}P^n_{\epsilon}\left(\textbf{u}\right)=\textbf{u}^*$. So, $\exists Orb_{\textbf{g}}\left(\textbf{u}^*\right)$ closed and stable.
These concludes the proof of Theorem \ref{theoremRegular}. 
\end{proof}
Let the Poincar\'{e} section $S=\left\{\textbf{u}\in H:\ qy_2>0\right\}$. So, by \ref{theoremRegular}, there is a first return map $P^1_{\epsilon}:\Omega\subset S\rightarrow\Omega$ with a unique stable fixed point $\textbf{u}^*_1\neq\textbf{0}$ ($\textbf{0}\notin\Omega$). So, $\exists P:\Omega\rightarrow\Omega$ in (\ref{reduced}) such that $P$ has a unique stable fixed point $\textbf{x}^*\neq\textbf{0}$ due continuity of the first return map in $H_0$ and  $\exists Orb_{\textbf{f}}\left(\textbf{x}^*\right)$.    
These concludes the proof of Theorem \ref{teocilinder}.
\end{proof}
\section{Coments About Singular Perturbed Vector Field}
In this section we did not study the dynamic in $\left|y_2\right|<\epsilon_0$ in (\ref{regular}). It easy to show that this vector field has singularity in $H^{\epsilon}_0$.
\begin{proposition}
$\exists\textbf{u}^*\in H^{\epsilon}_0:\ \textbf{g}\left(\textbf{u}^*\right)=\textbf{0}$, type hiperbolic saddle for almost all transition funciton $\varphi$.
\end{proposition}        
\begin{proof}
Let $\left[\textbf{u}\right]=\left[y_1\ y_2\ x_3\right]^t$. So, $\textbf{g}\left(\textbf{u}\right)=\textbf{0}\Rightarrow$ $y_2=0$, $y_1=\epsilon x_3$ and $\varphi\left(x_3\right)=x_3$. Let $\phi\left(x\right)=\varphi\left(x\right)-x$. Note that $\phi\left(-1\right)=\phi\left(1\right)=0$ and $\phi'\left(-1\right)=\phi'\left(1\right)=-1$. So, for $0<\delta<<1$ we have that $\phi\left(\delta-1\right)<0$ and $\phi\left(1-\delta\right)>0$. By intermidity value theorem $\exists x^*\in \left.\right]-1,1\left[\right.:\ \phi\left(x^*\right)=0$.

Note that the linearization of the (\ref{regular}) is $\epsilon\textbf{f}^-\left(\textbf{u}\right)+\alpha\left[y_2+\epsilon\left(\varphi'\left(x^*_3\right)-1\right)x_3\right]\frac{\partial\  }{\partial x_3}+o\left(\epsilon^2\right)$. We proved that the singularity $\textbf{u}^*$ is locally assintotically stable in the set $\left\{\textbf{u}\in H^{\epsilon}_0: x_3=0\right\}$.

As $\phi\left(\delta-1\right)<0$ and $\phi\left(1-\delta\right)>0$ and $\phi\in\mathcal{C}^{\infty}$ $\exists\delta_0>0:\ \forall x\in\left.\right]x^*-\delta_0,x^*\left[\right.\ \phi\left(x\right)\leq0\Rightarrow\varphi\left(x\right)\leq x$ and $\forall x\in\left.\right]x^*,\delta_0+x^*\left[\right.\ \phi\left(x\right)\geq0\Rightarrow\varphi\left(x\right)\geq x$. So $\frac{\varphi\left(x^*\right)-\varphi\left(x\right)}{x^*-x}\geq1\ \forall x\in\left.\right]x^*-\delta_0,x^*\left[\right.\cup\left.\right]x^*,\delta_0+x^*\left[\right.\Rightarrow\varphi'\left(x^*\right)\geq1$. Note that is $\varphi'\left(x^*\right)=1$ this point is not hiperbolic singulatity. So $\exists\varphi_1$ perturbation of the transiction function $\varphi$ such that $\varphi_1'\left(x^*_3\right)>1$.   
\end{proof}

This saddle singularitites are linked with homoclinic structures in regular Matsumoto-Chua system (\cite{Hirsch:2003}). In singular case, the saddle singularitites change to a pseudossingulaty. So, the homoclinic structures are destroyed.   
\section{Conclusion}

The regular Matsumoto-Chua system has homoclinic structures in the separatrix. Here, these structures are a unstable set. The separatrix is a border of the atractions basis of the equilibrium points and closed orbit.

In singular case, the homoclinic structures are destroyed. The closed orbit is preserved because it is disjoint of the escaping set. If this orbit has point in border of escaping set there is a bifurcation and this orbit disappears.

\section{Numerical simulations}

We have numerically integrated system 
using a fourth order Runge-Kutta method with step $10^{-2}$.
In figure \ref{sim1} the orbit has initial condition $\left[\textbf{x}_0\right]=\left[10,10,0\right]^t$ and control parameters $\alpha=5$ and $\beta=5$. It is consistent with the existences of an attraction region and of the minimal set type closed orbit described in theorem \ref{teocilinder} proof. 

\begin{figure}[h!]
\centering
\includegraphics[scale=0.5]{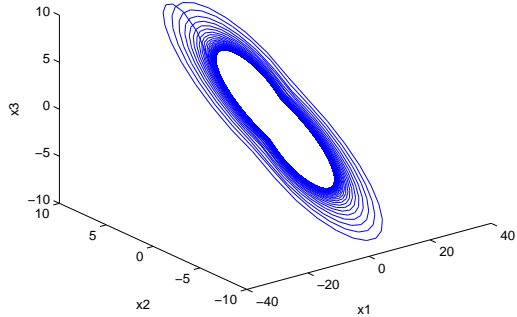}
\caption{Numerical integration of system for $\alpha=\beta=5$} 
\label{sim1}
\end{figure}

Figure \ref{sim2} illustrates the case without closed orbit, that has the same dynamics as stated in theorem \ref{teocilinder}. In this case, $\alpha=0,25$.

\begin{figure}[h!]
\centering
\includegraphics[scale=0.5]{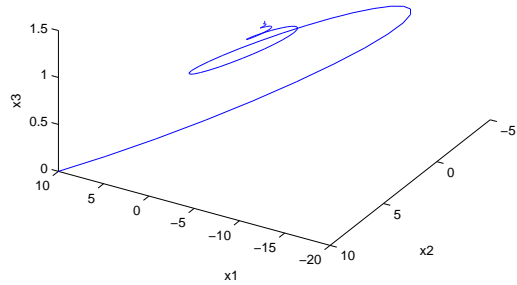}
\caption{Numerical integration of system for $\alpha=0,25$ and $\beta=5$} 
\label{sim2}
\end{figure}

\bibliographystyle{unsrt}
\bibliography{bibfile}
\end{document}